\documentclass[a4paper,10pt,twoside]{amsart}

\usepackage[utf8]{inputenc}
\usepackage{amsmath, amsfonts, amssymb,amsthm}
\usepackage[mathscr]{eucal} 
\usepackage[pdftex]{graphicx}
\usepackage{color}
\usepackage{multicol}

\usepackage[T1]{fontenc}
\usepackage[sc]{mathpazo}
\usepackage{enumerate}

\newcommand{\R}{\mathbb{R}}

\newcommand{\s}{\mathbb{S}}
\newcommand{\h}{\mathbb{H}}

\newcommand{\Ric}{\mathrm{Ric}}
\newcommand{\RP}{\R\mathrm{P}}

\newcommand{\df}{\,\mathrm{d}}

\newcommand{\vol}{\operatorname{area}}
\newcommand{\prodesc}[2]{\left\langle #1, #2 \right \rangle}
\newcommand{\abs}[1]{\left\lvert #1 \right\rvert}
\newcommand*{\defeq}{\mathrel{\vcenter{\baselineskip0.5ex \lineskiplimit0pt
                     \hbox{\scriptsize.}\hbox{\scriptsize.}}}%
                     =}
\DeclareMathOperator{\arccot}{arccot}

\newtheorem{theorem}{Theorem}
\newtheorem*{theorem*}{Theorem}
\newtheorem{proposition}{Proposition}

\newtheorem{lemma}{Lemma}

\theoremstyle{definition}

\theoremstyle{remark}
  \newtheorem{remark}{Remark}

\numberwithin{equation}{section}

\title[Compact embedded minimal surfaces in $\s^2\times\s^1$]{Compact embedded minimal surfaces in $\s^2\times\s^1$}

\author{José M. Manzano}
\address{Department of Geometry and Topology\\
University of Granada \\
Avda. Fuentenueva s/n, 18071, Granada, Spain}
\email{jmmanzano@ugr.es}

\author{Julia Plehnert}
\address{Discrete Differential Geometry Lab\\University of G\"ottingen\\Lotzestr. 16-18\\ 37083 G\"ottingen, Germany}
\email{j.plehnert@math.uni-goettingen.de}

\author{Francisco Torralbo}
\address{Department of Geometry and Topology\\
University of Granada \\
Avda. Fuentenueva s/n, 18071, Granada, Spain}
\email{ftorralbo@ugr.es}

\thanks{The first and third authors are partially supported by the Spanish MCyT-Feder research projects MTM2007-61775 and MTM2011-22547, as well as the Junta Andaluc\'{i}a Grants P06-FQM-01642 and P09-FQM-4496}


\begin{document}

\begin{abstract}
We prove that closed surfaces of all topological types, except for the non-orientable odd-genus ones, can be minimally embedded in $\s^2\times\s^1(r)$, for arbitrary radius $r$. We illustrate it by obtaining some periodic minimal surfaces in $\s^2\times\R$ via conjugate constructions. The resulting surfaces can be seen as the analogy to the Schwarz P-surface in these homogeneous $3$-manifolds.
\end{abstract}

\maketitle

\section{Introduction}

The geometry and topology of $3$-manifolds with non-negative curvature  has been an active field of research during the last century. One of the properties of these $3$-manifolds is the fact that  they generally admit compact (without boundary) embedded minimal surfaces.  Many authors have contributed to the study of compact minimal surfaces in order to understand the geometry of the $3$-manifold (see, for instance,~\cite{AR,GR,MSY, SY}). In 1970,  Lawson~\cite{Lawson} proved that any compact orientable surface can be minimally embedded in the constant sectional curvature $3$-sphere $\s^3$. Lawson's result has also been extended to different $3$-manifolds such as the Berger spheres (see~\cite{Torralbo}).

The aim of this paper is to determine which compact surfaces admit minimal embeddings in the homogeneous product $3$-manifolds $\s^2\times\s^1(r)$, $r>0$, where $\s^2$ denotes the constant curvature one sphere, and $\s^1(r)$ is the radius $r$ circle, i.e., $\s^1(r)$ has length $2\pi r$ and will be identified with $\R/2\pi r$. These $3$-manifolds have a distinguished unit vector field $\xi_{(p,t)} = (0, 1)$, which is parallel and generates the \emph{fibers} of the natural Riemannian submersion $\pi\colon  \s^2 \times \s^1(r) \rightarrow \s^2$. Given a vector field $X$ we will say that $X$ is \emph{horizontal} if it is orthogonal to $\xi$ and \emph{vertical} if it is parallel to $\xi$. Also, given a differentiable curve $\alpha$ in $\s^2\times \s^1(r)$ we will say that it is \emph{horizontal} (resp.\ \emph{vertical}) if its tangent vector is pointwise horizontal (resp.\ vertical). We can express the curvature tensor in terms of the metric and $\xi$ as
\begin{equation}\label{eq:curvature}
\begin{split}
R(X, Y)Z = &\langle Y, Z \rangle X - \langle X, Z\rangle Y + \langle X, \xi\rangle \langle Z, \xi \rangle Y - \langle Y, \xi\rangle \langle Z, \xi \rangle X + \\
&(\langle X, Z \rangle \langle Y, \xi\rangle - \langle Y, Z \rangle \langle X, \xi \rangle )\xi,
\end{split}
\end{equation}
and hence $\Ric(X) = \|X\|^2 - \langle X, \xi\rangle^2\geq 0$ for each vector field $X$.

We will begin by briefly discussing known examples with low genus. Recall that a non-orientable compact surface $\Sigma$ has (non-orientable) genus $g$ if it is topologically equivalent to the connected sum of $g$ projective planes (equivalently, $\chi(\Sigma) = 2 - g$, where $\chi(\Sigma)$ stands for the Euler characteristic of the surface $\Sigma$).

First of all, the only spheres minimally immersed in $\s^2\times\s^1(r)$ are the horizontal slices $\s^2\times\{t_0\}$ (it follows from lifting such a sphere to a minimal sphere in $\s ^2\times\R$ and applying the maximum principle with respect to a foliation of $\s ^2\times\R$ by slices), which are always embedded. This implies that the projective plane cannot be minimally immersed in $\s^2\times\s^1(r)$ for any $r>0$. Moreover, they are the only compact minimal stable surfaces in $\s^2 \times \s^1(r)$ (cf.~\cite[Corollary~1]{TU13}).

If the Euler characteristic is zero the situation is quite different in both orientable and non-orientable cases:
\begin{enumerate}[$\bullet$]
  \item On the one hand, \emph{vertical helicoids} in $\s^2\times\R$ form a $1$-parameter family of minimal surfaces ruled by ho\-ri\-zon\-tal geo\-de\-sics invariant under a $1$-parameter group of ambient screw-motions, see~\cite{Rosenberg}. Each vertical helicoid is determined by its \emph{pitch} $\ell \in [0, +\infty[$ (i.e., twice the distance between successive sheets of the helicoids). Hence, taking quotients by the vertical translation $(p,t)\mapsto(p,t+2\pi r)$ in $\s^2\times\R$, vertical helicoids with appropriate pitches give rise to embedded minimal tori ($\ell = 2\pi r$) and Klein bottles ($\ell = 4\pi r$) in $\s^2\times\s^1(r)$ for any $r > 0$. If $\ell \to +\infty$, the helicoid converges to the \emph{vertical cylinder} $\Gamma \times \R$, $\Gamma$ being a geodesic of $\s^2$.

  \item On the other hand, one can consider examples foliated by circles: if the centers of the circles lie in the same vertical axis $\{p_0\} \times \R$, then they are the rotationally invariant examples studied by Pedrosa and Ritoré~\cite{PR99}. They are called \emph{unduloids} by analogy to the Delaunay surfaces and form a $1$-parameter family of minimal surfaces in $\s^2\times\R$. Moreover, they are singly periodic in the $\R$ factor with period $T \in \,]0, 2\pi[$ (if $T \to 0$ the unduloid converges to a double cover of a slice, whereas if $T \to 2\pi$ the unduloid converges to the vertical cylinder).

If the centers of the circle do not lie in the same vertical axis, we get the Riemann-type examples studied by Hauswirth~\cite[Theorem 4.1]{Haus06}. They are also singly periodic and produce embedded minimal tori in $\s^2\times \s^1(r)$ for all $r > 0$ (see also~\cite[\S 7]{Ros03} for a beautiful description).
Note that, thanks to \cite[Theorem 4.8]{Urb12}, a compact minimal surface $\Sigma$ immersed in $\s^2 \times \s^1(r)$, $r \geq 1$, foliated by circles has index greater than or equal to one (the index is one only for $r=1$ and $\Sigma=\Gamma\times\s^1(1)$, $\Gamma\subset\s^2$  being a great circle). A similar estimation of the index for $r < 1$ remains an open question (see~\cite[Remark 4.9]{Urb12}).
\end{enumerate}

In the higher genus case, Rosenberg~\cite{Rosenberg} constructed compact minimal surfaces in $\s^2\times\s^1(r)$, for all $r>0$, by a technique similar to Lawson's. More specifically, he defined a polygon in $\s^2\times\R$ made out of vertical and horizontal geodesics depending on two parameters $\gamma=\frac{\pi}{d}$, $d\in\mathbb{N}$, and $\tilde{h}>0$ (see Figure~\ref{f:boundary1}). Then he solved the Plateau problem with respect to this contour and obtained a minimal disk that can be reflected across its boundary in virtue of Schwarz's reflection principle. This procedure gives a complete surface that projects to the quotient $\s^2\times\s^1(r)$, $r=\frac{\tilde{h}}{\pi}$, as a compact embedded minimal surface. Unfortunately, it fails to be orientable as stated in~\cite{Rosenberg}, being a non-orientable surface of Euler characteristic $\chi=2(1-d)$ so its (non-orientable) genus $k=2-\chi=2d$ is always even. One can solve this issue by considering $r=2\frac{\tilde{h}}{\pi}$ (i.e., doubling the vertical length, which represents a $2$-fold cover of the non-orientable examples), so the resulting minimal surface has Euler characteristic $\chi=4(1-d)$ and its (orientable) genus is $g=1-\frac{\chi}{2}=2d-1$, which turns out to be odd. When $d = 2m$ is even, Rosenberg's examples induce one-sided compact non-orientable embedded minimal surfaces in the quotient $\RP^2 \times \s^1(r)$ of odd genus $1+2m$, $m \geq 1$.

Very recently, Hoffman, Traizet and White obtained a class of properly embedded minimal surfaces in $\s^2\times\R$ called \emph{periodic} genus $g$ helicoids. More explicitly, given a genus $g\geq 1$, a radius $r>0$, and a helicoid $H\subset\s^2\times\R$ containing the horizontal geodesics $\Gamma\times\{0\}$ and $\Gamma\times\{\pm\pi r\}$ for some great circle $\Gamma\subset\s^2$, \cite[Theorem 1]{HTW} yields the existence of two compact orientable embedded minimal surfaces $M^+$ and $M^-$ with genus $g$ in the quotient of $\s^2\times\R$ by the translation $(p,t)\mapsto(p,t+2\pi r)$. The existence of infinitely-many non-congruent helicoids $H$ satisfying the conditions above implies the existence of infinitely-many non-congruent compact orientable embedded minimal surfaces in $\s^2\times\s^1(r)$ with genus $g$. The even genera examples of Hoffman-Traizet-White produce compact minimal surfaces in the quotient $\RP^2 \times \s^1(r)$, for all $r>0$.

Finally, it is worth mentioning the work of Coutant~\cite[Theorem 1.0.2]{Coutant12}, where Riemann-Wei type minimal surfaces in $\s^2\times \R$ are constructed.  Some of them give rise to compact embedded orientable minimal surfaces of arbitrary genus $g \geq 4$ in the quotient $\s^2 \times \s^1(r)$ for $r$ small enough.

In Section~\ref{sec:non-existence} we will state the main result of the paper proving that there are no compact embedded minimal surfaces in $\s^2\times\s^1(r)$ of the remaining topological types (i.e., non-orientable odd-genus surfaces) for any $r>0$:
\begin{quote}
\emph{Every compact surface but the odd Euler characteristic ones can be minimally embedded in $\s^2\times \s^1(r)$ for any $r > 0$.}
\end{quote}
We will also discuss some interesting topological properties of compact minimal surfaces in $\s^2\times\R$ that will lead to the non-existence result.

  Although we have shown that there exist compact embedded minimal examples of all allowed topological types in $\s^2\times\s^1(r)$, Section~\ref{sec:examples} will be devoted to obtain Schwarz P-type examples in $\s^2\times\s^1(r)$. More precisely:

  \begin{enumerate}[$\bullet$]
    \item \emph{Triply-periodic odd-genus orientable examples invariant by the isometry group of tilings of the sphere and a vertical translation (see Figures~\ref{f:cube-tessellation} and~\ref{f:balloon-tessellation})}.

    \item \emph{Arbitrary genus orientable surfaces in $\s^2\times \s^1(r)$ for $r$ small enough (cf.\ Proposition~\ref{prop:orientable-arbitrary-genus-examples}).}
  \end{enumerate}

All the examples will be constructed via Lawson's technique~\cite{Lawson}, and complementing it by considering the conjugate minimal surface (see~\cite{HST}) instead of the original solution to the Plateau problem for a geodesic polygon. In this case, it is well known that the boundary lines of the conjugate surface are curves of planar symmetry since the initial surface is bounded by geodesic curves. Again, reflecting the conjugate surface across its edges will produce a complete example. This technique, known as \emph{conjugate Plateau construction} has been applied in different situations (e.g., see~\cite{MT,Plehnert,Plehnert2}).

It is worth mentioning that a similar construction leads to classical P-Schwarz minimal surfaces in $\R^3$, but our construction can be also implemented in the product space $\h^2\times\R$  to produce embedded triply-periodic minimal surfaces with the symmetries of a tessellation of $\h^2$ by regular polygons together with a $1$-parameter group of vertical translations. We would also like to mention the minimal surfaces of $\h^2\times \R$ constructed in~\cite{MRR} also by similar techniques, specially the ones showed in~\cite[Section~6]{MRR} that are invariant by the isometry group of a regular tiling of $\h^2$.

We would like to thank Harold Rosenberg for his valuable comments that helped us improve this paper, as well as the referee, whose thorough report has encouraged us to clarify some aspects of the original manuscript.

\section{Topological non-existence results}\label{sec:non-existence}

We will begin by describing the intersection of two compact minimal surfaces in $\s^2\times\s^1(r)$. The ideas in the proof are adapted from those of Frankel~\cite{Frankel} (see also~\cite{AR,GR}). In~\cite[Theorem 4.3]{Rosenberg}, a similar result is proved for properly embedded minimal surfaces in $\s^2\times\R$ by using a different approach.

\begin{proposition}\label{prop:intersection}
Let $\Sigma_1$ and $\Sigma_2$ be two compact minimal surfaces immersed in $\s^2\times\s^1(r)$. If $\Sigma_1\cap\Sigma_2=\emptyset$, then $\Sigma_1$ and $\Sigma_2$ are two horizontal slices.
\end{proposition}

\begin{proof}
Let $\gamma\colon [a,b]\to\s^2\times\s^1(r)$ be a unit-speed curve satisfying $\gamma(a)\in\Sigma_1$, $\gamma(b)\in\Sigma_2$ and minimizing the distance from $\Sigma_1$ to $\Sigma_2$. This guarantees that $\gamma$ is orthogonal to $\Sigma_1$ at $\gamma(a)$ and to $\Sigma_2$ at $\gamma(b)$. For any parallel vector field $X$ along $\gamma$ and orthogonal to $\gamma'$, we can produce a variation $\gamma_t$ of $\gamma=\gamma_0$ with variational field $X$. Note that $\gamma_t$ can be chosen so that $\gamma_t(a)\in\Sigma_1$ and $\gamma_t(b)\in\Sigma_2$ for all $t$, since $X$ is orthogonal to $\gamma'$. Let $\{e_1, e_2, \gamma'(a)\}$ an orthonormal reference at $T_{\gamma(a)}\s^2\times \s^1(r)$ and $X_j$ the parallel transported vector field along $\gamma$ with $X_j(\gamma(a)) = e_j$, $j = 1, 2$. The function $\ell_X(t)=\mathrm{Length}(\gamma_t)$ satisfies $\ell_X'(0)=0$ and
\begin{equation}\label{eqn:second-variation}
\ell''_{X_1}(0) + \ell''_{X_2}(0) = H_1 - H_2 -\int_\gamma\Ric(\gamma') = -\int_\gamma \Ric(\gamma'),
\end{equation}
where $H_j$ is the mean curvature of $\Sigma_j$, $j = 1, 2$ (see~\cite[pages~69-70]{Frankel}).

The minimization property tells us that $\ell''_{X_1}(0) + \ell''_{X_2}(0)\geq 0$. Since  $\Ric(Y)\geq 0$ for any vector field $Y$ in $\s^2\times\s^1(r)$ and $\Ric(Y)=0$ if and only if $Y$ is vertical (see~\eqref{eq:curvature}), we deduce from equation~\eqref{eqn:second-variation} that $\gamma'$ is vertical. This means that the distance $d$ from $\Sigma_1$ to $\Sigma_2$ is realized by a vertical geodesic. Hence, translating $\Sigma_1$ vertically by distance $d$, we produce a contact point at which the translated surface is locally at one side of $\Sigma_2$. The maximum principle for minimal surfaces tells us that $\Sigma_2$ coincides with the vertical translated copy of $\Sigma_1$ at distance $d$.

Finally, we will suppose that $\Sigma_1$ is not horizontal at some point $p$ and reach a contradiction. In that case, note that the previous argument guarantees the existence of a vertical segment of distance $d$ joining $p$ with a point in $\Sigma_2$, but, since this segment is not orthogonal to any of the two surfaces, we could shorten it and provide a curve from $\Sigma_1$  to $\Sigma_2$ whose length is strictly less than $d$, which is a contradiction.
\end{proof}

As a consequence of this result, any compact embedded minimal surface in $\Sigma\subset\s^2\times\s^1(r)$ is either connected or a finite union of horizontal slices. Moreover, the lift of a connected compact embedded minimal surface of $\s^2\times \s^1(r)$ different from an horizontal slice to $\s^2\times\R$ is connected and periodic.

Let us now study the orientability of a compact minimal surface in $\s^2\times\s^1(r)$.

\begin{lemma}\label{lema:separation-S2xR}
Let $\Sigma\subset\s^2\times\R$ be a connected properly embedded minimal surface. Then $(\s^2\times\R)\smallsetminus\Sigma$ has two connected components and $\Sigma$ is orientable.
\end{lemma}

\begin{proof}
If $\Sigma$ is a horizontal slice, the result is trivial. Otherwise, given $t\in\R$, let us consider $S_t=\s^2\times\{t\}$. Then the set $\Sigma\cap S_t$ is the (non-empty) intersection of two minimal surfaces, hence it consists of a equiangular system of curves in the $2$-sphere $S_t$. We will denote by $C_t$ the set of intersection points of these curves (possibly empty), and it is well-known that an even number of such curves meet at each point of $C_t$. The fact that $\Sigma$ is properly embedded guarantees that $\cup_{t\in\R}C_t$ consists of isolated points.

We can decompose $S_t\smallsetminus\Sigma=A^1_t\cup A^2_t$ in such a way that $A^1_t,A^2_t\subset S_t$ are open and, given $i\in\{1,2\}$, the intersection of the closures of two connected components of $A^i_t$, for $i = 1, 2$, is contained in $C_t$. In other words, we are painting the components of $S_t\smallsetminus\Sigma$ in two colors so that adjacent components have different color.
Observe that the sets $\Sigma\cap S_t$ depend continuously on $t\in\R$ so it is clear that $A^1_t$ and $A^2_t$ can be chosen in such a way that $W_i=\cup_{t\in\R}A^i_t$ is open for $i\in\{1,2\}$. As $W_1\cap W_2=\emptyset$ and $W_1\cup W_2=(\s^2\times\R)\smallsetminus\Sigma$, we get that $W_1$ and $W_2$ are the connected components of $(\s^2\times\R)\smallsetminus\Sigma$. In particular, $\Sigma$ is orientable.
\end{proof}

If $\Sigma$ is a connected surface embedded in a orientable $3$-manifold $M$ and $M\smallsetminus\Sigma$ has two connected components, then $\Sigma$ is well-known to be orientable (the surface $\Sigma$ is said to \emph{separate} $M$). The converse is false in $\s^2\times\s^1(r)$ as horizontal slices show, but they turn out to be the only minimal counterexamples.

\begin{proposition}\label{lema:separation-S2xS1}
Let $\Sigma\subset\s^2\times\s^1(r)$ be a compact embedded minimal surface, different from a finite union of horizontal slices. Then $\Sigma$ separates $\s^2\times\s^1(r)$ if and only if $\Sigma$ is orientable.
\end{proposition}

\begin{proof}
We will suppose that $\Sigma$ is orientable and prove that it separates $\s^2\times\s^1(r)$. In order to achieve this, we consider the projection $\pi\colon \s^2\times\R\to\s^2\times\s^1(r)$ and the lifted surface $\widetilde\Sigma\subset\s^2\times\R$ such that $\pi(\widetilde\Sigma)=\Sigma$. Then $\widetilde\Sigma$ is properly embedded, and connected as a consequence of Proposition~\ref{prop:intersection}. Lemma~\ref{lema:separation-S2xR} states that we can decompose it in connected components $(\s^2\times\R)\smallsetminus\widetilde\Sigma=W_1\cup W_2$.

Let $\phi_r\colon \s^2\times\R\to\s^2\times\R$ be the vertical translation $\phi_r(p,t)=(p,t+2\pi r)$. As $\phi_r(\widetilde\Sigma)=\widetilde\Sigma$ there are two possible cases, $\phi_r$ either preserves $W_1$ and $W_2$ or swaps them:
\begin{enumerate}[(1)]
 \item If $\phi_r(W_1)=W_1$ and $\phi_r(W_2)=W_2$, then $\pi(W_1)\cap\pi(W_2)=\emptyset$. In this case $\Sigma$ separates $\s^2\times\s^1(r)$.
 \item If $\phi_r(W_1)=W_2$ and $\phi_r(W_2)=W_1$, then $\pi(W_1)=\pi(W_2)=(\s^2\times\s^1(r))\smallsetminus\Sigma$. In particular, $(\s^2\times\s^1(r))\smallsetminus\Sigma$ is connected and $\Sigma$ does not separate $\s^2\times\s^1(r)$. Since $\phi_r$ swaps $W_1$ and $W_2$ the unit normal vector field $\tilde{N}$ of $\tilde{\Sigma}$ does not induce a unit normal vector field on $\Sigma$. This contradicts the fact that $\Sigma$ is orientable because $\tilde{N}$ must be the lift of the unit normal vector field $N$ of $\Sigma$.\qedhere
\end{enumerate}
\end{proof}

We now state our main result, which is inspired by~\cite[Proposition 2]{Ros}.
\begin{theorem}
Let $\Sigma$ be a compact embedded non-orientable minimal surface  in $\s^2\times\s^1(r)$. Then $\Sigma$ has an even (non-orientable) genus.
\end{theorem}

\begin{proof}
We can lift $\Sigma$ in a natural way to a compact minimal surface $\Sigma_2\subset\s^2\times\s^1(2r)$ so $\Sigma_2$ is a $2$-fold cover of $\Sigma$. We also lift $\Sigma$ to a minimal surface $\widetilde\Sigma\subset\s^2\times\R$ and decompose $(\s^2\times\R)\smallsetminus\widetilde\Sigma=W_1\cup W_2$ as in the proof of Proposition~\ref{lema:separation-S2xS1}. The map $\phi_r\colon \s^2\times\R\to\s^2\times\R$ given by $\phi_r(p,t)=(p,t+2\pi r)$ swaps $W_1$ and $W_2$ so $\phi_{2r}\defeq\phi_r\circ\phi_r$ preserves $W_1$ and $W_2$, which means that $\Sigma_2$ is orientable. We will prove that $\Sigma_2$ has odd orientable genus, from where it follows that $\Sigma$ has even non-orientable genus.

Let us consider the isometric involution $F\colon \s^2\times\s^1(2r)\to\s^2\times\s^1(2r)$ given by $F(p,t)=(p,t+2\pi r)$, and the projection $\pi_2\colon \s^2\times\R\to\s^2\times\s^1(2r)$ such that $\pi_2(\widetilde\Sigma)=\Sigma_2$. It satisfies $F(\Sigma_2)=\Sigma_2$ and swaps $\pi_2(W_1)$ and $\pi_2(W_2)$. Given a horizontal slice $S\subset\s^2\times\s^1(2r)$ intersecting $\Sigma_2$ transversally, the surface $\Sigma_2\smallsetminus(S\cup F(S))$ can be split in two isometric surfaces $\Sigma_2'$ and $F(\Sigma_2')$. Decomposing $\Sigma_2'$ in connected components $P_1,\ldots,P_k$, we can compute $\chi(P_i)=2-2g_i-r_i$, where $g_i$ is the genus and $r_i$ the number of boundary components of $P_i$. Thus
\begin{equation}\label{eqn:genus}
\chi(\Sigma_2)=2\chi(\Sigma_2')=2\sum_{i=1}^k\chi(P_i)=4\left(k-\sum_{i=1}^kg_i\right)-2\sum_{i=1}^kr_i.
\end{equation}
The last term in equation~\eqref{eqn:genus} is a multiple of $4$ since the total number of boundary components $\sum_{i=1}^kr_i$ is even (note that each of such components appears twice, once in $S$ and once in $F(S)$), say $4m=\chi(\Sigma_2)=2(1-g)$. It follows that $g$, the genus of $\Sigma_2$, is odd and we are done.
\end{proof}

\section{Construction of examples}\label{sec:examples}

	Given an isometric minimal immersion $\tilde{\phi}: \Sigma \rightarrow \s^2\times \R$ of a simply connected surface $\Sigma$, Hauswirth, Sa Earp and Toubiana~\cite{HST} constructed an associated isometric minimal immersion $\phi: \Sigma \to \s^2\times \R$ called the \emph{conjugate immersion}. The following properties are well-known and will be used repeatedly in the sequel  (see~\cite{Daniel07} and~\cite[Lemma~1]{MT}):
\begin{enumerate}[(i)]
		\item \label{lm:properties-conjugate:item:angle-function}
		$\phi$ and $\tilde{\phi}$ have the same angle function, i.e., $\nu = \prodesc{N}{\xi} = \prodesc{\tilde{N}}{\xi}$, where $N$ and $\tilde{N}$ are the unit normal vector fields to $\phi$ and $\tilde{\phi}$ respectively.

		\item \label{lm:properties-conjugate:item:tangent-projection}
		The tangential projections of $\xi$ are rotated: $\df \phi^{-1}(T)=J\df\tilde{\phi}^{-1}(\tilde{T})$, where $T=\xi-\nu N$, $\tilde{T}=\xi-\nu \tilde{N}$, and $J$ is the $\tfrac{\pi}{2}$-rotation in $T\Sigma$.

		\item \label{lm:properties-conjugate:item:shape-operator}
		The shape operators $S$ and $\tilde{S}$ of $\phi$ and $\tilde{\phi}$ respectively are related by $S = J\tilde{S}$, where $J$ is the $\tfrac{\pi}{2}$-rotation in $T\Sigma$.

		\item \label{lm:properties-conjugate:item:symmetry-curves}
		Any geodesic curvature line in the initial surface becomes a planar line of symmetry in the conjugate one. More precisely, given a curve $\alpha$ in $\Sigma$, if $\tilde{\phi}\circ\alpha$ is a horizontal (resp.\ vertical) geodesic, then $\phi\circ\alpha$ is contained in a vertical plane (resp.\ slice), which the immersions meets orthogonally.
\end{enumerate}

	\begin{lemma}\label{lm:properties-conjugate}
		Let $\alpha$ be a curve in $\Sigma$ such that $\tilde{\gamma} = \tilde{\phi}\circ\alpha$ is a vertical unit-speed geodesic in $\s^2\times \R$, and express $\tilde{N}_{\tilde{\gamma}(s)} = \cos (\theta(s)) P_{\tilde{\gamma}(s)} + \sin (\theta(s)) Q_{\tilde{\gamma}(s)}$, where $\{P,Q,\tilde{\gamma}'\}$ is an orthonormal frame of parallel vector fields along $\tilde{\gamma}$. Then the conjugate curve $\gamma = \phi\circ\alpha$ lies in a slice $\s^2\times \{t_0\}$ and its geodesic curvature $\kappa$ as a curve of $\s^2\times \{t_0\}$ satisfies $\kappa(s) = \theta'(s)$.
	\end{lemma}

	\begin{proof}

 The Levi-Civita connection of $\s^2\times \R$ and the slice $\s^2\times \{t_0\}$ where $\gamma$ lies will be denoted by $\overline{\nabla}$ and $\nabla^{\s^2\times \{t_0\}}$, respectively. Since $\tilde{\gamma}$ is a vertical geodesic, the normal vector field $\tilde{N}$ along $\tilde{\gamma}$ is always horizontal. Hence, we can write $\tilde{N}_{\tilde{\gamma}(s)} = \cos (\theta(s)) P_{\tilde{\gamma}(s)} + \sin (\theta(s)) Q_{\tilde{\gamma}(s)}$ for some smooth function $\theta$. Hence
		\[
			\overline{\nabla}_{\tilde{\gamma}'} \tilde{N} = \theta' \bigl( -\sin(\theta) P + \cos(\theta) Q\bigr) = \theta' \tilde\phi_*(J\alpha'),
		\]
		because $P$ and $Q$ are parallel vector fields and $-\sin (\theta) P + \cos(\theta) Q$ is tangent to $\tilde{\phi}(\Sigma)$ (orthogonal to $N$) and orthogonal to $\tilde{\gamma}'$ (both $P$ and $Q$ are horizontal while $\tilde{\gamma}'$ is vertical by assumption) so it is colinear with $\tilde\phi_*(J\alpha')$ and, up to a change of orientation in $\Sigma$ if necessary, we can suppose that they are equal. Finally
		\[
		\begin{split}
			\theta' &= \prodesc{\overline{\nabla}_{\tilde{\gamma}'} \tilde{N}}{\tilde\phi_*(J\alpha')} = -\prodesc{\tilde{S}\alpha'}{J\alpha'} = \prodesc{J\tilde{S}\alpha'}{\alpha'} = \prodesc{S \alpha'}{\alpha'} =\\
&= \prodesc{\overline{\nabla}_{\gamma'} \gamma'}{N} = \prodesc{\nabla^{\s^2\times\{t_0\}}_{\gamma'}\gamma'}{N} = \kappa,
		\end{split}
		\]
		where we have used that $J$ is skew-adjoint and the properties of the conjugation stated above.
	\end{proof}

In this section we are going to apply the \emph{conjugate Plateau technique} in order to construct compact embedded minimal surfaces in $\s^2 \times \s^1(r)$. This technique consists in solving the Plateau problem over a geodesic polygon producing a minimal surface in $\s^2\times \R$, and considering its \emph{conjugate surface}. We must ensure that, after successive reflections over its boundary, we obtain a periodic (in the $\R$ factor) minimal surface, giving rise to a compact one in the quotient $\s^2 \times \s^1(r)$.

\subsection{Initial minimal piece and conjugate surface}
\label{subsec:contours}

Consider a geodesic triangle with angles $\tilde{\alpha}$, $\tilde{\beta}$ and $\gamma$ in $\s^2\times\{0\}$, lift the hinge given by the angle $\gamma$ in vertical direction and add two vertical geodesics (each of length $\tilde{h} > 0$) to produce a closed curve $\tilde{\Gamma}$ in $\s^2 \times \R$ (see Figure~\ref{f:boundary1}). For $\tilde{h}\geq0$, $0<\gamma<\pi$ and $\tilde{\alpha},\tilde{\beta}\leq\frac{\pi}{2}$ (so $\tilde{a}, \tilde{b}\leq \tfrac{\pi}{2}$)  the polygonal Jordan curve $\tilde{\Gamma}$ bounds a minimal graph $\tilde{M}$ over $\s^2 \times \{0\}$ (the existence follows from Radó's theorem, see~\cite{Rosenberg} for a particular example). Hence the angle function of $\tilde{M}$ does not vanish and we can suppose, up to a change of the normal to $\tilde{M}$, that $\tilde{\nu} > 0$. Its conjugate surface $M$ is a minimal surface also in $\s^2\times\R$, bounded by a closed Jordan curve $\Gamma$ consisting of three symmetry curves in vertical planes and two horizontal symmetry curves. Two successive curves in vertical planes enclose an angle $\gamma$ (this angle is intrinsic to the surface), the remaining angles in the vertexes of $\Gamma$ are equal to $\frac{\pi}{2}$.

	We will denote by $\tilde{\imath}$ (resp.\ $i$), $i \in \{1, 2, 3, 4, 5\}$, each of the vertexes of $\tilde{\Gamma}$ (resp.\ $\Gamma$). Moreover, we will denote $\widetilde{\imath\jmath}$ (resp.\ $\overline{\imath\jmath}$) the segment of $\tilde{\Gamma}$ (resp.\ $\Gamma$) that joints the points $\tilde{\imath}$ and $\tilde{\jmath}$ (resp.\ $i$ and $j$). Furthermore, $\Pi_{ij}$ will stand for the vertical plane or horizontal slice where $\overline{\imath\jmath}$ lies in the conjugate piece. Up to a translation in $\s^2\times \R$, we can suppose that the slice $\Pi_{34}$ is $\s^2\times \{0\}$.

	\begin{lemma}\label{lm:basic-properties-conjugate-polygon}
		In the previous setting, let $\tilde{h} \leq \tfrac{\pi}{2}$. Then: 
		\begin{enumerate}[(i)]
			\item The segments of $\Gamma$ are embedded curves. Moreover, each of the segment of $\Gamma$ projects injectively to $\s^2\times\{0\} = \Pi_{34}$.

			\item The vertical planes $\Pi_{23}$, $\Pi_{45}$ and $\Pi_{51}$ and the slices $\Pi_{12}$ and $\Pi_{34}$ determine a prism $\Omega$ defined by the angles $\alpha$ (between $\Pi_{51}$ and $\Pi_{23}$), $\beta$ (between $\Pi_{23}$ and $\Pi_{45}$), $\gamma$, and the height $h$ (see Figure~\ref{f:boundary2}). Moreover, $\alpha > \tilde{\alpha}$ and $\beta > \tilde{\beta}$.

			\item The polygon $\Gamma$ is contained in $\Omega$ and so $\Gamma$ projects injectively to $\s^2\times \{0\}$.

			\item The minimal surface $M$ is embedded and lies in the interior of $\Omega$.
		\end{enumerate}
	\end{lemma}

	Thanks to Lemma~\ref{lm:basic-properties-conjugate-polygon} we are able to draw more precisely what the polygon $\Gamma$ looks like (see Figure~\ref{f:boundary2}) provided $\tilde{h}$ is short enough, i.e., $\tilde{h}\leq \tfrac{\pi}{2}$.

\begin{figure}[htbp]
\begin{minipage}[b]{.5\textwidth}
  \begin{center}
  \includegraphics{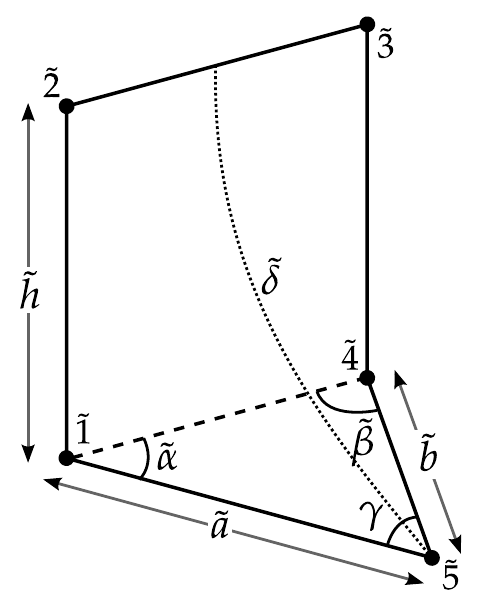}
  \end{center}
  \caption{\small{The boundary curve $\tilde{\Gamma}$. If $\tilde{a} = \tilde{b}$, then $\tilde{\Gamma}$ and $\tilde{M}$ are symmetric with respect to a vertical plane containing $\tilde{\delta}$.}}
  \label{f:boundary1}
\end{minipage}
\begin{minipage}[b]{.49\textwidth}
  \begin{center}
  \includegraphics{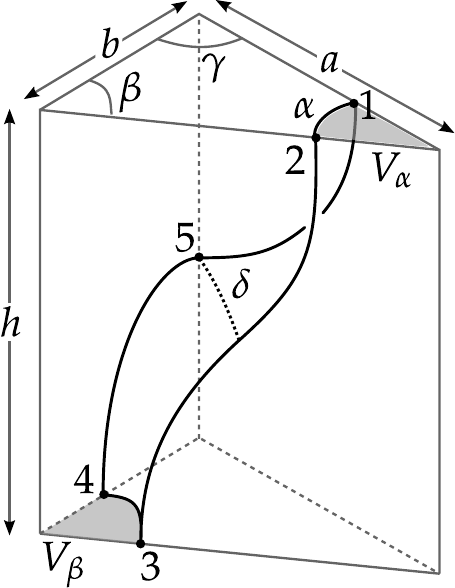}
  \end{center}
  \caption{\small{The boundary curve $\Gamma$ of the conjugate minimal surface $M$ inside the prism $\Omega$ with data $(\alpha, \beta, \gamma, h)$.}}\label{f:boundary2}
\end{minipage}
\end{figure}
	\begin{proof}
		(i) The angle function $\tilde{\nu}$ does not vanish in the interior of the horizontal segments of $\tilde{\Gamma}$ because this would contradict the boundary maximum principle for minimal surfaces ($\tilde\Sigma$ would be tangent to a vertical plane containing this horizontal segment). Hence the angle function $\nu$ of the conjugate piece (which is the same as $\tilde{\nu}$) does not vanish in the interior points of the curves contained in vertical planes. Let us focus on the segment $\overline{23}$ parametrized by a curve $\sigma:[a, b] \rightarrow \s^2\times \R$ with unit speed. The length of the projection $\pi(\overline{23})$ between the point $2 = \sigma(a)$ and $\sigma(t)$ is given by $\ell(t) = \int_a^t \nu(\sigma(s)) \df s$. Hence $\ell(t)$ is strictly increasing so the curve $\overline{23}$ projects injectively to $\s^2\times \{0\}$. In particular, it is embedded. The same argument works for $\overline{45}$ and $\overline{51}$.

The curves $\overline{12}$ and $\overline{34}$ are also embedded. To prove it, we realize that Lemma~\ref{lm:properties-conjugate} implies that the geodesic curvature of both curves does not change sign because the rotation of the normal vector field $\tilde{N}$ along $\widetilde{12}$ and $\widetilde{34}$ is monotonic (otherwise a contradiction with the boundary maximum principle would be found). Let us suppose that the curve $\overline{12}$ is not embedded. Then the curve contains a loop enclosing a domain $D$. By the Gau\ss-Bonnet theorem, 
		\[
		\mathrm{area}(D)\geq \pi - \int_{\text{loop}} \kappa \geq \pi - \abs{\int_{\text{loop}} \kappa }\geq \pi - \tilde{\alpha} \geq \frac{\pi}{2}, \quad \text{since } \tilde{\alpha} \in\ ]0,\tfrac{\pi}{2}], 
		\]
		where $\kappa$ is the geodesic curvature of $\overline{12}$ in $\Pi_{12}$ and we have taken into account that total curvature of $\overline{12}$ is exactly $\tilde{\alpha}$ (see Lemma~\ref{lm:properties-conjugate}). Hence, by the isoperimetric inequality, $\tilde{h} = \ell(\overline{12}) \geq \ell(\mathrm{loop}) \geq (\tfrac{7}{4}\pi^2)^{1/2} > \tfrac{\pi}{2}$, contradicting the assumption $\tilde{h}<\frac{\pi}{2}$. A similar reasoning implies that the curve $\overline{34}$ is embedded for $\tilde{h} < \tfrac{\pi}{2}$.

		(ii) First, the slices $\Pi_{12}$ and $\Pi_{34}$ do not coincide. Otherwise, we get a contradiction to the maximum principle by considering a slice $\s^2\times\{t\}$ for $|t|$ large enough so $(\s^2\times \{t\})\cap M = \emptyset$. Move such slice downwards or upwards until a first contact point appears. That point cannot be at the interior so it has to be either in $\overline{45}$, $\overline{51}$ or $\overline{23}$, and it would yield a contradiction to the boundary maximum principle). Let $h$ be the distance between $\Pi_{12}$ and $\Pi_{34}$ so $\Pi_{12} = \s^2\times \{h\}$. 

Note also that $\Pi_{45}$ and $\Pi_{51}$ are different since they form an angle of $\gamma \not\in\{0,\pi\}$. Finally, we will prove that $\Pi_{23}$ is different from $\Pi_{45}$ (resp.\ $\Pi_{51}$). Were it not the case, the endpoints of the horizontal curve $\sigma = \overline{34}$ (resp.\ $\sigma = \overline{12}$) would lie in a geodesic of $\s^2\times\{0\}$ (resp.\ $\s^2\times\{h\}$) meeting it orthogonally. Hence, the curve $\sigma$ and the segment of geodesic that joints the points $3$ and $4$ (resp.\ $1$ and $2$) would define a domain $D$ in $\s^2\times\{0\}$ (resp.\ $\s^2\times \{h\}$). Gau\ss-Bonnet formula in $D$ yields
		\begin{equation}\label{eqn:geod}
			\int_{\sigma} \kappa = \pi - \mathrm{area}(D),
		\end{equation}
		where $\kappa$ is the geodesic curvature of $\sigma$ in the slice.  The total rotation $\Theta$ of the normal vector field to $\tilde{M}$ along $\tilde{\sigma} = \widetilde{34}$ (resp. $\tilde{\sigma} = \widetilde{12}$) is such that $|\Theta|= \tilde{\beta}$ (resp.\ $|\Theta| = \tilde{\alpha}$), so $|\Theta|\leq\frac{\pi}{2}$. By Lemma~\ref{lm:properties-conjugate} and equation~\eqref{eqn:geod}, it follows that $\tfrac{\pi}{2}\leq\mathrm{area}(D) <\frac{3\pi}{2}$. Again the isoperimetric inequality in the sphere would imply that $4\tilde{h}^2 \geq \mathrm{area}(D)(4\pi - \mathrm{area}(D)) > \tfrac{7}{4}\pi^2$, contradicting that $\tilde h<\frac{\pi}{2}$.

		The last assertion of (ii) will be showed at the end of the proof.

		(iii) First, the height of the point $5$ is strictly between $0$ and $h$. In other case, for instance if the height of the point $5$ is bigger than or equal to $h$, we consider a slice $S = \s^2 \times \{t_0\}$ with $t_0$ sufficiently large so $M \cap S = \emptyset$. Then, we move $S$ downwards until it first touches $M$ at a point $p$. Then $p$ should be either an interior point of $M$, or at the interior of $\overline{45}$, $\overline{51}$ or $\overline{23}$, or $p = 5$. Anyway we will find a contradiction to the maximum principle at the interior or at the boundary. Likewise the curves $\overline{45}$, $\overline{51}$ and $\overline{23}$ are also contained in $\s^2\times[0,h]$. 

		Since the vertical planes $\Pi_{23}$, $\Pi_{45}$ and $\Pi_{51}$ are pairwise different (see (ii)), they divide the slice $\s^2\times \{0\}$ in eight triangles. Hence the vertical planes $\Pi_{23}$, $\Pi_{45}$ and $\Pi_{51}$ and the slices $\Pi_{12}$ and $\Pi_{34}$ define eight different prism. Since we have shown in (i) that the curves contained in vertical planes project one-to-one, (iii) will follow from the fact that $\overline{12}$ and $\overline{34}$ lie in the same prism, that we will call $\Omega$.

		\noindent\textbf{Claim 1}: \emph{The points $1$ and $2$ (resp.\ $3$ and $4$) lie in the same triangle.} 

Otherwise the geodesic curvature of $\overline{12}$ (resp.\ $\overline{34}$) in $\s^2\times \{h\}$ (resp.\ $\s^2\times \{0\}$) will change sign but that is impossible (see the proof of (i)).

		\noindent\textbf{Claim 2}: \emph{The curve $\overline{12}$ (resp.\ $\overline{34}$) is completely contained in a triangle.}

This claim follows from the fact that the geodesic curvature of $\overline{12}$ (resp.\ $\overline{34}$) does not change sign, as well as the fact the angle function of $\Sigma$ does not vanish. This means that the projection of $\Sigma$ to the $\s^2$-factor is an immersed domain $G\subset\s^2$, so we can define a unit normal to $\Sigma$ along $\overline{12}$ (resp.\ $\overline{34}$) that project to a normal to $G$ pointing towards the interior of $G$. If $\overline{12}$ or $\overline{34}$ are not contained in a triangle, it is easy to see that such an immersed domain $G$ cannot exist, which is a contradiction.

		\noindent\textbf{Claim 3}: \emph{The curves $\overline{12}$ and $\overline{34}$ lie in the same prism.} 

Let us assume that $\overline{12}$ and $\overline{34}$ do not lie in the same prism and suppose, without loss of generality, that $\overline{12}$ does not lie in the same prism as the point $5$. In that case the surface $M$ will contain two different points in the same vertical geodesic $\Pi_{51}\cap \Pi_{23}$. Hence there is one point $p$ in the interior of $M$ such that $\nu(p) = 0$ which is impossible (recall that $\nu = \tilde{\nu} > 0$, see the beginning of Section~\ref{subsec:contours}, and that $M$ is orthogonal to $\Pi_{51}$ and $\Pi_{23}$).

	Once we know that $\Gamma$ is contained in the boundary of a well-defined prism $\Omega$ (see (ii)), (i) ensures that the whole polygon $\Gamma$ projects injectively to $\s^2\times\{0\}$.

		(iv) This assertion follows from (iii) by a classical application of the maximum principle.

		It remains to prove that $\alpha > \tilde{\alpha}$ and $\beta > \tilde{\beta}$. First, let us apply the Gau\ss{}-Bonnet-theorem to the domain $V_\alpha$ bounded by $\overline{12}$ and the edges of $\Omega$ as in Figure~\ref{f:boundary2} (it is well defined thanks to (i) and (iii)). We get $\alpha=\vol(V_\alpha)+\tilde{\alpha}>\tilde{\alpha}$. Likewise, $\beta>\tilde\beta$ by using the corresponding domain $V_\beta$. 
	\end{proof}

Our aim is to define $\tilde{\Gamma}$ such that the conjugate surface has the desired properties, i.e., it can be smoothly extended without branch points by Schwarz reflection about its boundary producing a compact embedded surface in $\s^2\times\R$ after a finite number of reflections. Depending on the surface we want to construct we have to ensure different data $(\alpha,\beta,\gamma, h)$ for the prism $\Omega$ (see Lemma~\ref{lm:basic-properties-conjugate-polygon}). We want to remark that if $m$ copies of $M$ are needed to produce a compact orientable minimal surface $\Sigma$ by Schwarz reflection, then the genus $g$ of $\Sigma$ is $g = 1 + \frac{m}{4\pi}(\pi - \gamma)$. This can be worked out by the Gau\ss-Bonnet theorem: The total curvature of $M$ is $\gamma - \pi$ and in order to close the surface we need by assumption $m$ copies of $M$. Hence the total curvature of $\Sigma$ is $m(\gamma - \pi)$ and so the genus is $g = 1 + \frac{m}{4\pi}(\pi - \gamma)$.

\begin{remark}
One could consider the same boundary curve $\tilde{\Gamma}$ consisting of horizontal and vertical geodesics in $\R^3$. This leads to the well-known Schwarz P-surface, whose name is motivated by the fact that it is invariant under a primitive cubic lattice $\Lambda$. Moreover, there exists another lattice under which the surface is preserved but its orientation is not; $\Lambda$ is a subset of it. The quotient $P$ of the surface under the lattice $\Lambda$ has genus $3$ and consists of 16 copies of $M$ with $\alpha=\beta=\frac{\pi}{4}$ and $\gamma=\frac{\pi}{2}$. If we set $a=b=h$ (cp.\ Figure \ref{f:boundary2}), the quotient $P$ is contained in a cube of edge length $2a$, and these cubes tessellate $\R^3$.
\end{remark}

\subsection{Odd genus Schwarz P-type examples.}\label{subsec:odd-genus}

This section is devoted to investigate the most symmetric case in the construction above, which follows from choosing $\tilde a=\tilde b$ in the initial contour. Since the solution of the Plateau problem with boundary $\tilde\Gamma$ is unique and in this particular case $\tilde\Gamma$ is symmetric about a vertical plane, we get that the initial minimal surface $\tilde M$ is also symmetric about the same vertical plane, intersecting $\tilde M$ orthogonally along a curve $\tilde\delta$ (see Figure~\ref{f:boundary1}). From~\cite[\S2]{MT}, this implies that the conjugate surface $M$ is symmetric about a horizontal geodesic $\delta$ lying in the interior of $M$ (see Figure \ref{f:boundary2}). In particular, $a=b$ and $\alpha=\beta$.

\begin{figure}[h]
\begin{minipage}[b]{.5\textwidth}
\centering
\includegraphics{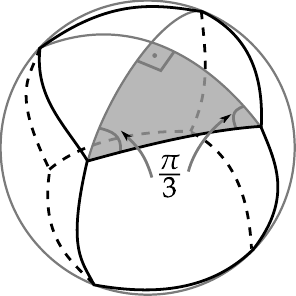}
\caption{\small{Regular tiling of $\s^2$ by quadrilaterals}}\label{f:cube-tessellation}
\end{minipage}
\begin{minipage}[b]{.49\textwidth}
  \centering
  \includegraphics{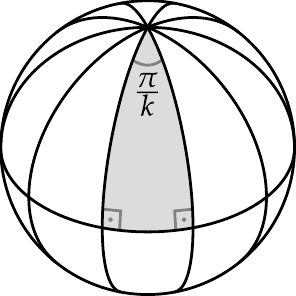}
  \caption{\small{Tessellation of the sphere by isosceles triangles}}\label{f:balloon-tessellation}
\end{minipage}
\end{figure}

We are going to produce new compact orientable embedded minimal examples in $\s^2 \times \s^1(r)$, for sufficiently small radius $r > 0$, coming from two different tessellations of the sphere by isosceles triangles (see the shaded triangles in Figures~\ref{f:cube-tessellation} and \ref{f:balloon-tessellation}). The first one comes from the regular quadrilateral tiling of the sphere once we decompose each square in four isosceles triangles with angles $\alpha = \beta = \frac{\pi}{3}$ and $\gamma = \frac{\pi}{2}$. The second one, as depicted in Figure~\ref{f:balloon-tessellation}, follows from decomposing $\s^2$ in $4k$ isosceles triangles with angles $\alpha = \beta = \frac{\pi}{2}$ and $\gamma = \frac{\pi}{k}$, for any $k\geq 1$. Let us assume that we can produce, via conjugate Plateau technique, an embedded minimal surface $M_j$ that fits inside a prism $\Omega_j$ with data $(\frac{\pi}{3}, \frac{\pi}{3},\frac{\pi}{2}, h)$ if $j = 1$ and $(\frac{\pi}{2}, \frac{\pi}{2}, \frac{\pi}{k}, h)$ if $j = 2$ for arbitrary small $h$ (see Figure~\ref{f:boundary2}). Then:

\begin{enumerate}[$\bullet$]
  \item It is clear that 24 congruent copies of $\Omega_1$ tessellate $\s^2\times [0,h]$. Since the tessellation comes from reflections about vertical planes, it implies a smooth continuation of $M_1$, which we call $M_1'$. The surface $M_1'$ is connected, topologically it is $\s^2\setminus\{p_1,\dots,p_8\}$, where each $p_i$ corresponds to one vertex of the quadrilateral tiling. After reflecting $M_1'$ about one of the horizontal planes $\s^2\times \{0\}$ or $\s^2\times \{h\}$ we get a surface consisting of 48 copies of $M_1$. The vertical translation $T$ about $2h$ yields a simply periodic embedded minimal surface $S$. It follows that the quotient $S/T$ is a compact minimal surface with genus $7$ in $\s^2 \times \s^1(\frac{h}{\pi})$.

  \item We can reflect the surface $M_2$ about the vertical planes containing the edges $a$ and $b$. Repeating this reflection successively the surface closes up in such a way that $2k$ copies of $M_2$ build a smooth surface in the product of an hemisphere of $\s^2$ and an interval $I$ of length $h$. The surface meets the vertical plane above the bounding great circle orthogonally and a reflection about this plane yields a minimal surface $M_2'$ which is topologically $\s^2\setminus\{p_1,\dots,p_{2k}\}$, where each $p_j$ corresponds to one vertex of the tiling. As before reflecting $M_2'$ about one of the horizontal planes $\s^2\times \partial I$ we get a surface $S$ which is invariant under the vertical translation $T$ of length $2h$. Moreover it is invariant under a rotation about the fiber over the north or south pole by an angle of $\frac{\pi}{k}$. The quotient $S/T$ is compact and consists of $8k$ copies of $M$, so its genus is given by $g=2k-1$, i.e., an arbitrary odd positive integer.
\end{enumerate}

Now, we are going to show the existence of the embedded minimal surfaces $M_1$ (resp.\ $M_2$) associated to $\Omega_1$ (resp.\ $\Omega_2$). Let $\gamma = \frac{\pi}{2}$ (resp.\ $\gamma = \frac{\pi}{k}$).
Recall that we have fixed $\tilde{a} = \tilde{b}$. As mentioned at the beginning of the section, the conjugate piece contains a horizontal geodesic $\delta$ (so $a = b$ and $\alpha = \beta$, see Figures~\ref{f:boundary1} and~\ref{f:boundary2}). Therefore, $\delta$ and its projection onto the base of $\Omega$ have the same length $\ell(\pi(\delta)) = \ell(\delta) = \ell(\tilde{\delta})$, and we can decompose the base of $\Omega$ in two right triangles as in Figure~\ref{fig:base-triangle-Omega}. By spherical trigonometry, we get that 
\begin{equation}\label{eq:relation-alpha-delta}
\cos(\alpha) = \sin(\tfrac{\gamma}{2}) \cos(\ell(\delta)).
\end{equation}

\begin{figure}[htbp]
\centering
\includegraphics{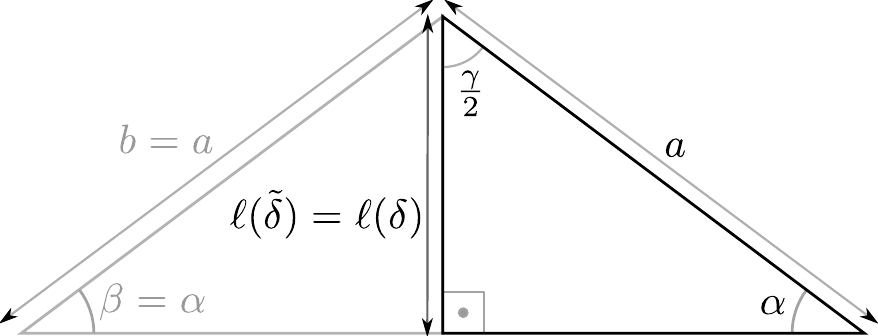}
\caption{Base triangle of the prism $\Omega$ in the special case $\tilde{a} = \tilde{b}$.}
\label{fig:base-triangle-Omega}
\end{figure}

We claim that \emph{for each $\tilde{a} \in\ ]0, \tfrac{\pi}{2}]$ there exists $\tilde{h}(\tilde{a}) \leq \tfrac{\pi}{2}$ such that $\alpha = \tfrac{\pi}{3}$ (resp.\ $\alpha = \tfrac{\pi}{2}$)}. Let us check the claim for $M_1$, i.e., in the case $\gamma = \tfrac{\pi}{2}$ and $\alpha = \tfrac{\pi}{3}$. Fixing $\tilde{a}\in\ ]0, \tfrac{\pi}{2}]$, the length of $\tilde{\delta}$ varies continuously in $\tilde{h}$ from zero (when $\tilde{h} \rightarrow 0$) to $+\infty$ (when $\tilde{h} \rightarrow +\infty$). Moreover, $\tilde{h} < \ell(\tilde{\delta})$ by construction. Since $\gamma = \tfrac{\pi}{2}$ and $\alpha = \tfrac{\pi}{3}$ then, in view of equation~\eqref{eq:relation-alpha-delta}, $\ell(\tilde\delta) = \ell(\delta) = \arccos(\sqrt{2}\cos(\tfrac{\pi}{3}))< \tfrac{\pi}{2}$ (in the case $\gamma=\tfrac{\pi}{k}$ and $\alpha = \tfrac{\pi}{2}$, we have that $\ell(\delta) = \tfrac{\pi}{2}$). Hence the height in the initial polygon satisfies $\tilde{h} \leq \ell(\tilde{\delta}) \leq \tfrac{\pi}{2}$ and the assumption of Lemma~\ref{lm:basic-properties-conjugate-polygon} is fulfilled. Finally, varying $\tilde{h} \in\ ]0, \ell(\tilde{\delta})[$ we can find $\tilde{h}(\tilde{a})$ such that $\alpha = \tfrac{\pi}{3}$ and $\tilde{h}(\tilde{a}) \leq \tfrac{\pi}{2}$. Lemma~\ref{lm:basic-properties-conjugate-polygon} applies and the claim is proved.

Since $h \leq \ell(\overline{23}) = \ell(\widetilde{23}) = 2 \arctan\bigl(\sin(\tilde{a})( {\cot^2(\tfrac{\gamma}{2}) + \cos^2(\tilde{a})})^{-1/2}\bigr)$, we can make $h$ arbitrarily small by choosing $\tilde{a}$ small enough.

To sum up, we can always choose a polygon $\tilde{\Gamma}$ with data $(\tilde{\alpha}, \tilde{\alpha}, \frac{\pi}{2}, \tilde{h})$ (resp.\ $(\tilde{\alpha}, \tilde{\alpha}, \frac{\pi}{k}, \tilde{h})$) satisfying the hypothesis of Lemma~\ref{lm:basic-properties-conjugate-polygon}. Thus the conjugate of the Plateau solution of $\tilde{\Gamma}$ is embedded and its boundary is contained in the prism $\Omega_1$ (resp.\ $\Omega_2$), for sufficiently small height $h$.

\subsection{Arbitrary genus Schwarz P-type examples.}\label{subsec:arbitrary-genus}
To construct a compact orient\-able minimal surface $M_k$ with arbitrary genus, let us guarantee the existence of a minimal surface $M$ that matches a prism $\Omega$ with data $(\frac\pi k,\frac\pi 2,\frac\pi 2,h)$ for some $h>0$. This leads to less symmetric examples since $\alpha\ne\beta$ and we have to use a degree argument as in \cite{KPS88} to guarantee existence. If $M\subset\Omega$ exists, then Schwarz reflection about the vertical mirror planes continues the surface smoothly and $4k$ copies of $M$ build a minimal surface $M'$ in the product $\s^2\times ]0,h[$ (see Figure~\ref{f:arbitrary-genus}). Topologically $M'$ is $\s^2\setminus\{p_1,\dots,p_{k+2}\}$, where $p_j$ are the vertexes of the tiling. As in the symmetric cases above, reflecting $M'$ about the bounding horizontal mirror planes gives a complete surface $M''$, which is invariant under vertical translation about $2h$. The quotient is the desired surface $M_k$. Moreover, $M''$ is also invariant under rotation about those vertical fibers which are intersection of vertical mirror planes by angles $\frac{2\pi}{k}$ (resp.\ $\pi$). As in the previous cases one computes the genus and gets $1+k$, since $8k$ copies of $M$ build the compact surface $M_k$.

\begin{figure}[htbp]
\centering
\includegraphics[height=4.5cm]{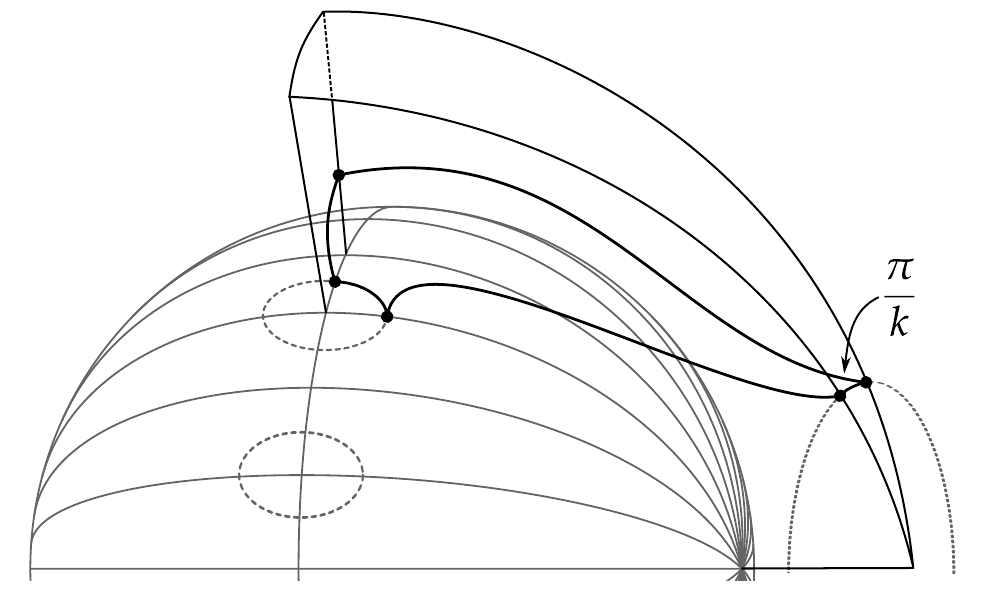}
\caption{\small{The polygon $\Gamma$ for the configuration $(\frac\pi k,\frac\pi 2,\frac\pi 2,h)$. The resulting compact surface $M_k$ has genus $k+1$.
}}
\label{f:arbitrary-genus}
\end{figure}

\begin{proposition}\label{prop:orientable-arbitrary-genus-examples}
There exists a compact orientable minimal surface $M_k$ with genus $k+1$, embedded in $\s^2\times\s^1(r)$, for any $k\geq3$ and $r$ small enough.
\end{proposition}

\begin{proof}
We only need to prove the existence of the minimal surface $M$ with data $(\frac\pi k,\frac\pi 2,\frac\pi 2,h)$, the existence of $M_k$ follows directly by Schwarz reflection.

We start with a polygonal Jordan curve $\tilde{\Gamma}$ as in Figure \ref{f:boundary1} with $\gamma=\frac{\pi}{2}$ and some $\tilde{h}>0$. We have seen that the minimal surface $M$ is then uniquely defined by $(\tilde{a},\tilde{b})$, so this defines a map $f\colon \R_{+}\times\,]0,\frac{\pi}{2}]^2\to \R^3,\,(\tilde{h},\tilde{a},\tilde{b})\mapsto(h,\alpha,\beta)$. The map $f$ is continuous, since a sequence $\{\tilde{M}_n\}$ of solutions  to the Plateau problem with data $(\tilde{h}_n,\tilde{a}_n,\tilde{b}_n)$, it has a converging subsequence with limit given by the solution to the Plateau problem with data $(\tilde{h}_\infty, \tilde{a}_\infty, \tilde{b}_\infty)=\lim_n(\tilde{h}_n,\tilde{a}_n,\tilde{b}_n)$. In order to justify this, observe that each $\tilde{M}_n$ can be extended by Schwarz reflection to a complete minimal surface $\hat{M}_n$. Since the ambient geometry is bounded and the surfaces have uniformly bounded second fundamental forms, general convergence arguments show that there exists a subsequence of $\{\hat{M}_n\}$ converging  to a complete minimal surface $\hat{M}_\infty$ in the $\mathcal{C}^k$-topology on compact subsets for every $k\geq 0$. Since every surface $\hat{M}_n$ contains the polygon $\tilde{\Gamma}_n$ associated to $(\tilde{h}_n,\tilde{a}_n,\tilde{b}_n)$, the limit surface $\hat{M}_\infty$ contains the polygon $\tilde{\Gamma}_\infty$ associated to $(\tilde{h}_\infty, \tilde{a}_\infty, \tilde{b}_\infty)$. Moreover, $\hat{M}_\infty$ is also a graph in the interior of the polygon, so it coincides with the solution to the Plateau problem with respect to $\tilde{\Gamma}_\infty$ (note that such a solution is unique).

In order to prove the existence of a minimal surface $M$, i.e., the existence of a triple $(\tilde{h}_0,\tilde{a}_0,\tilde{b}_0)$ such that $f(\tilde{h}_0,\tilde{a}_0,\tilde{b}_0)=(h_0,\frac\pi k,\frac\pi 2)$ for some $h_0>0$, we use a degree argument (recall that $\tilde{\alpha}, \tilde{\beta}\in\ ]0, \tfrac{\pi}{2}]$ so $\tilde{a}, \tilde{b} \in\ ]0, \tfrac{\pi}{2}]$). We consider the map 
\[
  f_{\tilde{h}}\colon]0,\tfrac\pi2]\times\ ]0,\tfrac\pi2]\to \R^2,\,(\tilde{a},\tilde{b})\mapsto(\alpha,\beta)
\]
and show there exists a closed Jordan curve $c\colon \R\to ]0,\frac\pi2]^2$ such that the image $f_{\tilde{h}}\circ c$ is a closed curve around $(\frac\pi k,\frac \pi 2)$. We compose $c$ of four straight lines $c_i$, $i\in\{1,\ldots, 4\}$, namely,
\begin{multicols}{2}
\begin{itemize}
\item $c_1(t)=\!(\frac\pi2,t)$ with $t\in[\frac{1}{2k},\frac{\pi}{k}]$,
\item $c_2(t)=(t,\frac{\pi}{k})$ with $t\in[\frac12,\frac\pi2]$,
\item $c_3(t)=(\frac12,t)$ with $t\in[\frac{1}{2k},\frac\pi k]$,
\item $c_4(t)=(t,\frac1{2k})$ with $t\in[\frac12,\frac\pi2]$,
\end{itemize}
\end{multicols}
\noindent and we claim there exists $\tilde{h}>0$ such that $f_{\tilde{h}}\circ c$ has the desired property.

Since $\tilde{\gamma}=\tilde{a}=\frac\pi2$ along $c_1$, we have $\beta>\tilde{\beta}=\frac\pi2$ (by the cosine rule $\cos(\tilde{\beta}) = 0$ since $\gamma = \tilde{a} = \tfrac{\pi}{2}$) along $f_{\tilde{h}}\circ c_1$ (see Lemma~\ref{lm:basic-properties-conjugate-polygon}). Likewise, $\alpha>\tilde\alpha>\frac\pi k$ along $f_{\tilde{h}}\circ c_2$ (by cosine rule $\cos(\tilde{\alpha})= \sin(\tilde{\beta})\cos(\tilde{b}) < \cos(\tilde{b}) = \cos(\tfrac{\pi}{k})$). Now, $\tilde{\beta} < \tfrac{\pi}{2}$ along $c_3$ since, by the cosine rule, $\cos(\tilde{\beta}) = \cos(\tilde{a}) \sin(\tilde{\alpha})  = \cos(\tfrac{1}{2})\sin(\tilde{\alpha})> 0$. Moreover, we will also show that $\tilde{\alpha} < \tfrac{\pi}{k}$ along $c_4$. First, by the sine rule, 
\[
\sin(\tilde{\beta}) = \frac{\sin(\tilde{\alpha})}{\sin(\tilde{b})}\sin{(\tilde{a})} > \frac{\sin(\tilde{\alpha})}{\sin(\tfrac{1}{2k})} \sin(\tfrac{1}{2}), \quad \text{since } \tilde{a} \in\  ]\tfrac{1}{2},\tfrac{\pi}{2}[ \text{ along $c_4$}.
\]
By the cosine rule, $\cos(\tilde{\alpha}) = \sin(\tilde{\beta})\cos(\tilde{b}) = \sin(\tilde{\beta})\cos(\tfrac{1}{2k})$. Hence
\[
\cos(\tilde{\alpha})> \sin(\tfrac{1}{2})\cot(\tfrac{1}{2k}) \sin(\tilde{\alpha}),
\]
so $\tilde{\alpha} < \arccot\bigl(\sin(\tfrac{1}{2})\cot(\tfrac{1}{2k})\bigr) < \tfrac{\pi}{k}$. Therefore, we have showed that $\tilde\beta<\frac\pi2$ (resp. $\tilde\alpha<\frac{\pi}{k}$) along $c_3$ (resp.\ $c_4$). Since $\alpha\to\tilde{\alpha}$ and $\beta\to\tilde{\beta}$ for $\tilde{h}\to 0$, we deduce that there exists $0 < \tilde{h}_0\leq \tfrac{\pi}{2}$ such that $\beta<\frac\pi 2$ along $f_{\tilde{h}}\circ c_3$ and $\alpha<\frac\pi k$ along $f_{\tilde{h}}\circ c_4$ for all $\tilde{h}\leq\tilde{h}_0$. Note that we can choose $\tilde{h}_0$ such that Lemma~\ref{lm:basic-properties-conjugate-polygon} applies and we are done.
\end{proof}

\begin{remark}
Note that the Schwarz P-type example of genus $2k-1$, $k\geq 2$, (see Section~\ref{subsec:odd-genus}) induces a compact embedded minimal surface in $\RP^2\times \s^1(r)$ if and only if $k = 2m$. In this case, we get a two-sided non-orientable embedded minimal surface of genus $4m$, $m \geq 1$.

The orientable arbitrary genus embedded minimal surface $M_k$, $k \geq 3$, constructed in Proposition~\ref{prop:orientable-arbitrary-genus-examples} induces a compact embedded minimal surface in $\RP^2\times \s^1(r)$ if and only if $k = 2m$. In this case, the induce surface is two-sided non-orientable embedded and minimal with genus $2(1+m)$, $m \geq 2$.
\end{remark}

\subsection{Final remarks} 
In this section we are going to point out the reason why it is difficult to obtain a better result in $\s^2 \times \s^1(r)$ using the Plateau construction technique for all $r$. First we state the following general properties:

\begin{proposition}\label{prop:symmetries}
\begin{enumerate}[(i)]
  \item Let $\Sigma$ be a properly embedded minimal surface of $\s^2 \times \R$. If $\Sigma$ contains a vertical geodesic, then it also contains the antipodal geodesic. More precisely if $\{p\} \times \R \subset \Sigma$, then $\{-p\}\times \R \subset \Sigma$.
  \item Let $\Sigma$ be an oriented, embedded minimal surface of $\s^2 \times \s^1(r)$ different from a finite union of horizontal slices. If $\Sigma$ contains a horizontal geodesic, then it also contains another horizontal geodesic at vertical distance $\pi r$. More precisely, if $\Gamma\times \{0\} \subset \Sigma$, $\Gamma$ a great circle of $\s^2$, then $\Gamma\times\{\pi r\} \subset \Sigma$.
\end{enumerate}
\end{proposition}

\begin{proof}
Both assertions follow from two facts: (1) under the hypothesis the surface $\Sigma$ separates the ambient space in two different connected components; (2) there exists an isometry $\rho$ of the ambient space which preserves the surface by the Schwarz reflection principle but interchanges the connected components of the complement of $\Sigma$. Hence, the fixed points of $\rho$ must be contained in $\Sigma$. In the first case, $\rho$ is the reflection around $\{p\} \times \R$ whereas in the second case $\rho$ is the reflection around $\Gamma \times \{0\}$. The first assertion appears in~\cite[\S 2]{HW}.
\end{proof}

One important consequence of this proposition is that, generically, the oriented compact embedded minimal surfaces constructed solving and reflecting a Plateau problem have odd genus. More precisely, let $M$ a minimal disk (different from an open subset of a slice) spanning a polygon $\Gamma$ made of horizontal and vertical geodesics. Suppose that $M$ produces a compact surface $\Sigma$ by Schwarz reflection, and the angle function $\nu = \prodesc{N}{\xi}$, where $N$ is the unit normal to $M$, satisfies $\nu^2\lvert_{M} < 1$ and $\nu^2\lvert_\Gamma = 1$ only at the points of $\Gamma$ where two horizontal geodesics meet. Then $\Sigma$ has odd genus if it is orientable or it has even genus if it is not.

In order to prove this, we consider the vector field $T = \xi - \nu N$ which is the tangent part of $\xi$, so $\abs{T}^2 = 1-\nu^2$, i.e., $T$ only vanishes in the points where $\nu^2 = 1$ that turns out to be the points where the Hopf differential associated to $\Sigma$ vanishes. The zeros of $T$ appear in multiples of $4$ since if $(p,0) \in \Sigma$ is such a point then it is an intersection point of two horizontal geodesics $\gamma_1 \times \{0\}$ and $\gamma_2\times \{0\}$ contained in the surface $\Sigma$ and so the point $(-p,0) \in (\gamma_1\cap \gamma_2) \times \{0\}$ is also in $\Sigma$ and $T_{(-p,0)} = 0$. Moreover, by Proposition~\ref{prop:symmetries}, the horizontal geodesics $\gamma_j \times\{\pi r\}$, $j = 1, 2$, are contained in $\Sigma$ so $(p, \pi r)$ and $(-p, \pi r)$ are also zeros of $T$. Applying the Poincaré-Hopf theorem to the vector field $T$, the Euler characteristic of $\Sigma$ must be $-4k$, that is, the genus of $\Sigma$ is $g = 2k + 1$ if $\Sigma$ is orientable or $g = 2(1+2k)$ if it is not.

Note that this trick is not useful in general for the conjugate surface of $M$ (we assume that the conjugate piece can be reflected to obtain a compact surface): Although the angle function is preserved by conjugation and so are the zeros of the vector field $T$, we can not guarantee in general that the zeros of $T$ occur in the intersection of two horizontal geodesics in the conjugate piece.


\begin{thebibliography}{[19]}
  \bibliographystyle{alpha}

\bibitem{AR}
M.\ T.\ Anderson and L.\ Rodríguez.
\newblock Minimal surfaces and $3$-manifolds of non-negative Ricci curvature.
\newblock \emph{Math.\ Ann.}, \textbf{284} (1989), no. 3, 461--475.

\bibitem{Coutant12}
A.~Coutant.
\newblock Déformation et construction de surfaces minimales.
\newblock PhD Thesis. Paris (2012).

\bibitem{Daniel07}
B.~Daniel.
\newblock Isometric immersions into 3-dimensional homogeneous manifolds.
\newblock {\em Comment. Math. Helv.}, \textbf{82} (2007), no. 1, 87--131.

\bibitem{Frankel}
T.~Frankel.
\newblock On the fundamental group of a compact minimal submanifold.
\newblock \emph{Ann.\ of Math. (2)}, \textbf{83} (1966), 68--73.

\bibitem{GR}
G.\ J.\ Galloway and L.\ Rodríguez.
\newblock Intersections of minimal surfaces.
\newblock \emph{Geom.\ Dedicata}, \textbf{39} (1991), no. 1, 29--42.

\bibitem{Haus06}
L.~Hauswirth.
\newblock Minimal surfaces of Riemann type in three-dimensional product manifolds.
\newblock \emph{Pacific J.\ Math.}, \textbf{224} (2006), no. 1, 91--118.

\bibitem{HST}
L.\ Hauswirth, R.\ Sa Earp, and E.\ Toubiana.
\newblock Associate and conjugate minimal immersions in $\boldsymbol{M} \times \boldsymbol{R}$.
\newblock  \emph{Tohoku Math.\ J. (2)}, \textbf{60} (2008), no. 2, 267--286.

\bibitem{HTW}
D.\ Hoffman, M.\ Traizet, and B.\ White.
\newblock Helicoidal minimal surfaces of prescribed genus, I.
\newblock Preprint available at arXiv:1304.5861 [math.DG].

\bibitem{HW}
D.\ Hoffman  and B.\ White.
\newblock Axial minimal surfaces in $\mathbf{S}^2\times\mathbf{R}$ are helicoidal.
\newblock \emph{J.\ Differential Geom.}, \textbf{87} (2011), no. 3, 515--523.

\bibitem{KPS88}
H.\ Karcher, U.\ Pinkall, and I.\ Sterling.
\newblock New minimal surfaces in ${S\sp 3}$.
\newblock \emph{J.\ Differential Geom.}, \textbf{28} (1988), no. 2, 169--185.

\bibitem{Lawson}
H.~B.~Lawson.
\newblock Complete minimal surfaces in {$S\sp{3}$}.
\newblock \emph{Ann.\ of Math. (2)}, \textbf{92} (1970), 335--374.

\bibitem{MT}
J.\ M.\ Manzano and F.\ Torralbo.
\newblock New examples of constant mean curvature surfaces in $\mathbb{S}^2\times\mathbb{R}$ and $\mathbb{H}^2\times\mathbb{R}$.
\newblock \emph{Michigan Math. J.}, \textbf{63} (2014), no. 4, 701--723.

\bibitem{MRR}
L.~Mazet, M.\ M.~Rodríguez and H.~Rosenberg.
\newblock Periodic constant mean curvature surfaces in $\h^2\times\R$.
\newblock \emph{Asian J. Math.}, \textbf{18} (2014), no. 5, 829--858.

\bibitem{MSY}
W.~H.~Meeks, L.~Simon, and S.-T.~Yau.
\newblock Embedded minimal surfaces, exotic spheres, and manifolds with positive Ricci curvature.
\newblock \emph{Ann.\ of Math. (2)}, \textbf{116} (1983), no. 3, 621--659.

\bibitem{PR99}
R.\ Pedrosa and M.\ Ritoré.
\newblock Isoperimetric domains in the Riemannian product of a circle with a simply connected space form and applications to free boundary problems.
\newblock {\em Indiana Univ.\ Math.\ J.}, \textbf{48} (1999), no. 4, 1357--1394.

\bibitem{Plehnert}
J.\ Plehnert.
\newblock Surfaces with constant mean curvature $1/2$ and genus $1$ in $\mathbb{H}^2 \times \mathbb{R}$.
\newblock Preprint available at arXiv:1212.2796 [math.DG].

\bibitem{Plehnert2}
J.\ Plehnert.
\newblock Constant mean curvature $k$-noids in homogeneous manifolds.
\newblock \emph{Illinois J.\ Math.} (to appear).

\bibitem{Ros}
A.\ Ros.
\newblock The Willmore conjecture in the real projective space.
\newblock \emph{Math.\ Res.\ Lett.} \textbf{6} (1999), no. 5--6, 487--494.

\bibitem{Rosenberg}
H.~Rosenberg.
\newblock Minimal surfaces in $\mathbb{M}^2\times\mathbb{R}$.
\newblock \emph{Illinois J.\ Math.}, \textbf{46} (2002), no. 4, 1177--1195.

\bibitem{Ros03}
H.~Rosenberg.
\newblock Some recent developments in the theory of minimal surfaces in 3-manifolds.
\newblock Publicações Matemáticas do IMPA. 24º Colóquio Brasileiro de Matemática. Instituto de Matemática Pura e Aplicada (IMPA), Rio de Janeiro, 2003. iv+48 pp. ISBN: 85-244-0210-5.

\bibitem{SY}
R.~Schoen and S.-T.~Yau.
\newblock Existence of incompressible minimal surfaces and the topology of three dimensional manifolds with non-negative scalar curvature.
\newblock \emph{Ann.\ of Math.\ (2)}, \textbf{110} (1979), no. 1, 127--142.

\bibitem{Torralbo}
F.\ Torralbo.
\newblock Compact minimal surfaces in the Berger spheres.
\newblock \emph{Ann.\ Global Anal.\ Geom.}, \textbf{41} (2012), no. 4, 391--405.

\bibitem{TU13}
F.~Torralbo and F.~Urbano.
\newblock On stable compact minimal submanifolds.
\newblock \emph{Proc.\ Amer.\ Math.\ Soc.}, \textbf{142} (2014), no. 2, 651--658.

\bibitem{Urb12}
F.~Urbano.
\newblock Second variation of one-sided complete minimal surfaces
\newblock \emph{Rev.\ Mat.\ Iberoamericana}, \textbf{29} (2013), no. 2, 479--494.

\end{thebibliography}
\end{document}